\documentclass[12pt]{article}

\usepackage{amssymb,amsmath, amsthm,latexsym, verbatim, amscd}
\usepackage{color}

\usepackage{geometry}
\usepackage[all]{xy}
\usepackage{graphicx}
\usepackage{xr}



\theoremstyle{plain}
\newtheorem{theorem}{Theorem}[section]
\newtheorem{lemma}[theorem]{Lemma}
\newtheorem{corollary}[theorem]{Corollary}
\newtheorem{proposition}[theorem]{Proposition}

\theoremstyle{definition}

\newtheorem{example}[theorem]{Example}

\theoremstyle{remark}
\newtheorem{remark}{Remark}

\usepackage{hyperref}

\usepackage{tikz}
\usetikzlibrary{shapes,arrows,patterns}

\numberwithin{equation}{section}

\begin{document}

\title{Double Gegenbauer expansion of ${\left|s{-} t\right|}^\alpha$}

\author{
T.~Kobayashi
\thanks{CONTACT T.~Kobayashi. Email: toshi@ms.u-tokyo.ac.jp}
\thanks{Graduate School of Mathematical Sciences,
	The University of Tokyo
     and Kavli Institute for the Physics and Mathematics of the Universe (WPI)}
and 
A.~Leontiev
\thanks{Graduate School of Mathematical Sciences,
	The University of Tokyo
}
}

\maketitle

\begin{abstract}
	  Motivated by the study of symmetry breaking operators for indefinite
  orthogonal groups, we give a Gegenbauer expansion of the two variable
  function $| s {-} t |^{\alpha}$ in terms of the ultraspherical polynomials
  $C_{\ell}^{\lambda} (s)$ and $C^{\mu}_m (t)$. 
  Generalization, specialization,
 and limits of the expansion are also discussed.
\end{abstract}


Gegenbauer polynomial; Sobolev inequality; Hermite polynomial; Selberg integral

\vskip 0.5pc
{\bf{AMS CLASSIFICATION}}


\section{Main results}
Let $C_{\ell}^{\lambda} (s)$ be the Gegenbauer polynomial of degree $\ell$. 
In this article, we give an expansion of the power ${\left|s{-} t\right|}^\alpha$ by two Gegenbauer polynomials $C^\lambda_\ell(s)$ and $C^\mu_m(t)$ with independent parameters $\lambda$ and $\mu$.
	
For $\ell,m\in\mathbb{N}$, we set
\begin{equation}
     \displaystyle 
{{b^{\lambda,\mu,\nu}_{\ell,m}}} :=\displaystyle \frac{{{(-1)^{m}}} (\lambda + \ell)
    (\mu + m) \Gamma (\lambda + \mu + 2 \nu + 1) \Gamma (\lambda) \Gamma (\mu)
    \Gamma (2 \nu + 1)}{2^{ 2 \nu}
    \displaystyle\prod_{\delta,\varepsilon\in\{\pm1\}}\Gamma\left( 
\nu+1+\frac{\lambda+\mu}{2}+\delta\frac{\lambda+\ell}{2}+\varepsilon
\frac{\mu+m}{2}
\right)}.
    \label{eqn:alm}
\end{equation}
\begin{theorem}
[{cf. Kobayashi--Mano \cite[Lem.~7.9.1]{kobayashi2011schrodinger}}]
	\label{thm:1-1}
	Let $\lambda,\mu$ be positive numbers and $\nu\in\mathbb{C}$ satisfying $2\operatorname{Re}\, \nu>\lambda+\mu+4$.
	For $\varepsilon=0,1$, we have an expansion
  \begin{equation}
		  | s {-} t |^{2 \nu}{\operatorname{sgn}}^\varepsilon(s{-} t) = \displaystyle\sum_{\scalebox{0.6}{$\begin{array}[]{c}
			  \ell,m\in\mathbb{N}\\
			  \ell+m\equiv\varepsilon{\,\operatorname{mod}\,2}
		  \end{array}$}} {{b^{\lambda,\mu,\nu}_{\ell,m}}} C_{\ell}^{\lambda}
    (s) C_m^{\mu} (t),
	  \label{eqn:s+t}
  \end{equation}
  where the right-hand side converges absolutely and uniformly in $[-1,1]^2$.
\end{theorem}
More generally, we find a double Gegenbauer expansion of
the function $(s-xt)_+^\alpha$ of two variables $s$ and $t$ with two parameters
$x\in[-1,1]$ and $\alpha\in\mathbb{C}$ with ${\operatorname{Re}\alpha}>0$. 
Here we 
recall that both $\left\{ {\left|y\right|}^\alpha,
{\left|y\right|}^\alpha{\operatorname{sgn}}{\left(y\right)}\right\}$
and $\left\{ y^\alpha_+,y^\alpha_- \right\}$
span the space of continuous homogeneous functions
on $\mathbb{R}$ of degree $\alpha$ when ${\operatorname{Re}\alpha}>0$,
and that the change of basis is given by
\begin{equation}
    {\left|y\right|}^\alpha{\operatorname{sgn}}^\varepsilon(y)=
    y_+^\alpha+(-1)^\varepsilon y^\alpha_-\quad
    (\varepsilon=0,1),
    \label{eqn:Riesz}
\end{equation}
where we set
\begin{equation*}
    y^\alpha_+:=
    \left\{\begin{array}[]{ll}
        y^\alpha&\left( y>0 \right)\\
        0&\left( y\le0 \right),
\end{array}\right.\quad
y_-^\alpha:=\left\{ \begin{array}[]{ll}
0&\left( y\ge0 \right)\\
{\left|y\right|}^\alpha&\left( y<0 \right).
\end{array}\right.
\end{equation*}
(The equation \eqref{eqn:Riesz} can be understood as the
identity of distributions with meromorphic
parameter $\alpha$, see \cite[Sect.~2]{clerc2011}
or \cite[Chap.~1]{gelfand1964}, although we do not
need this viewpoint here.)
 Then, we prove the following integral formula
(see also Proposition \ref{prop:4:6d1} for its variants):
\begin{theorem}
  \label{main-thm}
  For $\ell, m \in \mathbb{N}$,
  $\lambda, \mu, \nu \in \mathbb{C}$ with ${\operatorname{Re}\lambda},
  {\operatorname{Re}\mu}>-\frac{1}{2},{\operatorname{Re}\nu}>0$, and for $- 1 \leqslant x \leqslant
  1$, we set
  \begin{equation*}
	{{B^{\lambda,\mu,\nu}_{\ell,m}\left( x \right)}}:=
     \displaystyle
     \int_{- 1}^1 \displaystyle\int_{- 1}^1 {\left( s{-} xt \right)}_+^{2 \nu} u_{\ell}^{\lambda} (s)
     u_m^{\mu} (t) d s d t,
  \end{equation*}
  where
\[ u_{\ell}^{\lambda} (s) := \frac{2^{2 \lambda - 1} \ell ! \Gamma
   (\lambda)}{\Gamma (2 \lambda + \ell)} (1 - s^2)^{\lambda - \frac{1}{2}}
   C_{\ell}^{\lambda} (s) . \]
  Then we have
  
  \begin{equation}
	  {{B^{\lambda,\mu,\nu}_{\ell,m}\left( x \right)}}=\frac{
		  {{(-1)^{m}}} \pi^2\Gamma(2\nu+1)
     x^m{}_2 F_1
    \left( \begin{array}{c}
      - \nu+\frac{\ell + m}{2} , - \lambda- \nu +\frac{m - \ell}{2} \\
      \mu + m + 1
    \end{array} ; x^2 \right)
    }
    {2^{2\nu+1}
	    \Gamma\left( \nu-\frac{\ell+m}{2}+1 \right)
    \Gamma (\mu + m + 1) \Gamma \left( \lambda +
    \nu + \frac{\ell - m}{2} + 1 \right)
    }.
  \label{eqn:main+}
  \end{equation}
\end{theorem}
By the Sobolev-type estimate for the Gegenbauer expansion 
given in Proposition \ref{prop:1718105} and the
elementary identity \eqref{eqn:Riesz},
Theorem \ref{thm:1-1} is deduced
 readily from the special case at $x=1$ of Theorem \ref{main-thm}. As another application 
of Theorem \ref{main-thm}, we prove the following integral formula: we
set $d\omega^{\alpha,\beta}(x):= x^\alpha(1-x)^\beta dx$ and 
$d\omega^\alpha(x):=(1-x^2)^\alpha dx$.
\begin{corollary}
	\label{cor:CCint}
	Assume $\lambda,\mu,\nu,b\in\mathbb{C}$ and $\ell,m\in\mathbb{N}$ satisfy {$\ell+m\in2\mathbb{N}$},
		${\operatorname{Re}\lambda},{\operatorname{Re}\mu}>-\frac{1}{2},{\operatorname{Re}\nu}>0$, and 
		${\operatorname{Re}b}>-1$.
		Then $\frac{1}{2}
		(m-\ell)$ is an integer and we have
		
\begin{alignat*}{1}
&\int_{-1}^1\int_{-1}^1\int_0^1{\left|s-t\sqrt{y}\right|}^{2\nu}
C^\lambda_\ell(s)C^\mu_m(t)d\omega^{\mu+\frac{m}{2},b}(y)
d\omega^{\lambda-\frac{1}{2}}(s)
d\omega^{\mu-\frac{1}{2}}(t)\\
&=
{c}\cdot
\frac
{
(2\lambda)_\ell(2\mu)_m
\left( -\nu \right)_{\frac{\ell+m}{2}}
\Gamma\left( \lambda+\frac{1}{2} \right)\Gamma\left( \mu+\frac{1}{2} \right)
\Gamma \left( \nu+\frac{1}{2}\right)
\Gamma \left( { \lambda+\mu+2\nu+b+2} \right)
\Gamma \left( b+1 \right)
}
{
\Gamma \left(\lambda+\mu+\nu+b+\frac{\ell+m}{2}+2 \right) 
\Gamma \left( \lambda +\nu+\frac{ \ell- m   }{2}+1 \right) 
\Gamma \left(
\mu +\nu+b-\frac{  \ell-m }{2}+2 \right)
},
\end{alignat*}
	where $c$ is the nonzero constant 
	$\displaystyle
	\frac{\pi^{\frac{1}{2}}(-1)^{\frac{m-\ell}{2}}
	}{\ell!m!}$.
\end{corollary}
\begin{remark}
As in \eqref{eqn:Riesz}, we may take another basis for the space of 
continuous homogeneous functions on $\mathbb R$ of degree $\alpha$ given by
\begin{equation}
    (x \pm i 0)^\alpha =
    x_+^\alpha+e^{\pm \pi i \alpha} x^\alpha_-\quad.
    \label{eqn:Riesz2}
\end{equation}
By change of basis, 
 we can derive easily closed formul{\ae}
 of the double Gegenbauer expansion of $|s-t|_{\pm}^{2 \nu}$
 and $|s-t \pm i 0|^{2 \nu}$ from \eqref{eqn:alm}
 in Theorem \ref{thm:1-1}
 and {\it{vice versa}}.  
Similarly,
 we can find readily integral formul{\ae}
 for $(s-x t)_{-}^{2 \nu}$, 
 $|s- x t|^{2 \nu}$, 
 $|s-x t|^{2 \nu} {\operatorname{sgn}}(s-x t)$, 
 and $(s- x t \pm i 0)^{2 \nu}$
 analogous to Theorem \ref{main-thm}.  
Likewise for Corollary \ref{cor:CCint}.  
\end{remark}
	Selberg-type integrals are related to (finite-dimensional) representation
	theory of semisimple Lie algebras,
 see {\cite{forrester2008importance, tarasov2003selberg}} and references therein.
	On the other hand, (an equivalent form of) Theorem \ref{thm:1-1}  was given earlier by Kobayashi--Mano 
\cite[Lem.~7.9.1]{kobayashi2011schrodinger},
 which was utilized in the study of the 
{\textit{unitary inversion operator}} of the geometric quantization of the minimal nilpotent orbit.
 Furthermore, the precise location of 
	the zeros and
	the poles of the meromorphic continuation
	of the formul\ae{} given in
	Theorem \ref{main-thm} and Proposition \ref{cor:1} will be used in the study of symmetry
	breaking operators for infinite-dimensional representations when we extend the
	work {\cite{kobayashi2015symmetry}} on the rank one group
	to indefinite orthogonal groups $O (p,q)$ of higher rank. This will be given in a subsequent paper.

	The proof of Theorems \ref{thm:1-1} and \ref{main-thm} will be given in
	Sections \ref{sec:pfThm} and \ref{sec:2}, respectively.
        Corollary \ref{cor:CCint} is shown in Section \ref{sec:pfOfCol}.
	Special cases and the limit case of Theorem \ref{main-thm} will be discussed
	in Sections \ref{sec:4} and \ref{sec:limit}.

	Notation: $\mathbb{N}=\left\{ 0,1,2,\cdots \right
		\}$, $(x)_n= 
		x(x+1)\cdots(x+n-1)$ (the Pochhammer
		symbol), and $[\lambda]$ denotes
		the greatest
		integer that does not exceed
		$\lambda\in
		\mathbb{R}$.
	\section{Proof of the main theorem}\label{sec:2}
	In this section we prove that Theorem \ref{main-thm}
	is deduced from the special case $\ell=m=0$, namely, from
	the following
	integral formula \eqref{eqn:stz}.
	Proposition \ref{prop:2}
	will be proved in
	Section \ref{sec:3}.

	\begin{proposition}
		\label{prop:2}Suppose $a, b, c \in \mathbb{C}$ satisfy ${\operatorname{Re}a},
		{\operatorname{Re}b}> 0$ and ${\operatorname{Re}c}>\frac{1}{2}$. For $- 1 \leqslant x \leqslant 1$, we have
	
{\begin{alignat}{1} & {\displaystyle\int_{-1}^{1}\int_{-1}^{1}{\left( s-x{{}} t \right)_+^{2{{}} c-1}}{\left(1-s^2\right)^{a-1}}{\left(1-t^2\right)^{b-1}}{d}s{d}t}\nonumber\\=& \frac{\sqrt{\pi}{{}} {\Gamma\left( a \right)}{\Gamma\left( b \right)}{\Gamma\left( c \right)}}{2{{}} {\Gamma\left( a + c \right)}{\Gamma\left( b + \frac{1}{2}  \right)}}{{}}{{}_2F_1\left( \begin{array}[]{c}\displaystyle-c+\frac{1}{2},-a-c+1\\b+\frac{1}{2}\end{array};x^2\right)}.\label{eqn:stz}\end{alignat}}
\end{proposition}

\begin{proof}
[Proof of Theorem \ref{main-thm}]
The Rodrigues formula for the Gegenbauer polynomial (see
	  {\cite[(6.4.14)]{andrews2000special}} for instance) shows

{\begin{equation}
{u_{\ell}^\lambda(s)}=
          \frac{{\left(-1\right)^{\ell}}{}{2^{-\ell}}{}\sqrt{\pi}}{{\Gamma\left( 
              \lambda + \ell + \frac{1}{2} \right)}} \cdot {\frac{d^{\ell}}{ds^{\ell}}}
              {{\left(1-s^2\right)^{\lambda + \ell - \frac{1}{2}}}} . 
     \label{eqn:Rod}
\end{equation}}
By the definition of $B^{\lambda,\mu,\nu}_{\ell,m}(x)$,
	  the left-hand side of
	  \eqref{eqn:main+} amounts to
	  \begin{eqnarray}
		  & \displaystyle\frac{2^{- \ell - m} (-1)^{\ell+m}\pi}
		  {\Gamma \left( \lambda + \ell + \frac{1}{2}
	    \right) \Gamma \left( \mu + m + \frac{1}{2} \right)} I_{\ell, m} (x), & 
	    \nonumber
	  \end{eqnarray}
	  where we set
	  \[ I_{\ell, m}(x) \equiv I^{\lambda,\mu,\nu}_{\ell,m}(x) := \displaystyle\int_{- 1}^1 \displaystyle\int_{- 1}^1 \left( s{-} xt \right)_+
		  ^{2 \nu}
	     \frac{\partial^{\ell}}{\partial s^{\ell}} (1 - s^2)^{\lambda + \ell -
	     \frac{1}{2}} \frac{\partial^m}{\partial t^m} (1 - t^2)^{\mu + m -
	     \frac{1}{2}} d s d t. \]
Suppose ${\operatorname{Re}\nu} > \frac{\ell + m}{2}$, ${\operatorname{Re}\lambda} >
   \frac{1}{2}$ and ${\operatorname{Re}\mu} > \frac{1}{2}$. Then
	  integration by parts gives
\begin{align*}
			  I_{\ell, m} (x) &= (-1)^{\ell+m}\int_{- 1}^1 \displaystyle\int_{- 1}^1 
			  \left(\frac{\partial^{\ell +
	    m}}{\partial s^{\ell} \partial t^m} \left( s{-} xt \right)_+^{2 \nu}  \right)
			  (1 - s^2)^{\lambda + \ell -
	    \frac{1}{2}} (1 - t^2)^{\mu + m - \frac{1}{2}}  d s d t\\
		  &={{(-1)^{m}}} (- 2 \nu)_{\ell + m} x^m \displaystyle\int_{- 1}^1 \displaystyle\int_{- 1}^1 \left(s{-} xt \right)_+^{2 \nu -
	    \ell - m} (1 - s^2)^{\lambda + \ell - \frac{1}{2}} (1 - t^2)^{\mu + m -
	    \frac{1}{2}} d s d t,
\end{align*}
because
\[ \frac{\partial^{\ell + m}}{\partial s^{\ell} \partial t^m} \left( s{-} xt \right)_+^{2
	  \nu} =
	  (-1)^{\ell}
	  (- 2 \nu)_{\ell + m} x^m \left( s{-} xt \right)_+^{2 \nu - \ell - m}. \]
Applying Proposition \ref{prop:2}
 with $(a, b, c) = \left( \lambda + \ell + \frac{1}{2}, \mu + m + \frac{1}{2}, \nu +
  \frac{1}{2} (1 - \ell - m) \right)$, we see that
  the equation \eqref{eqn:main+} holds in the domain of
  $(\lambda,\mu,\nu)$ that we treated. Now Theorem \ref{main-thm} follows
  by analytic continuation.
\end{proof}

\section{Proof of Proposition \ref{prop:2}}\label{sec:3}
	In this section we show Proposition \ref{prop:2}.
	We use the following two lemmas.
	\begin{lemma}
		\label{lem4}For $a, b \in \mathbb{C}$ with ${\operatorname{Re}a}, {\operatorname{Re}b} >
	  0$ and for $- 1 \leqslant x \leqslant 1$ we have
	  
{\begin{equation}{\displaystyle\int_{-1}^{1}{\left(1-t{{}}x\right)^{a-1}}{\left(1-t^2\right)^{b-1}}{d}t}={B\left(\frac{1}{2},b\right)}{}{{}_2F_1\left( \begin{array}[]{c}\displaystyle\frac{1-a}{2},\frac{2-a}{2}\\b+\frac{1}{2}\end{array};x^2\right)}.\label{eqn:EulerRep}\end{equation}}
\end{lemma}
	\begin{lemma}
		\label{lem:Fisum}Let $a, b, d \in \mathbb{C}$
	with $b + \frac{1}{2} \notin
	  -\mathbb{N}$ and $d\notin-\mathbb{N}$. Then the series
	{\begin{equation*}G(a,b,d;\zeta):= {\sum_{i=0}^{\infty}\frac{{\left(a\right)_{i}}{}{\left(1-a\right)_{i}}}{{2^{i}}{}{i!}{}{\left(d\right)_{i}}}{}{{}_2F_1\left( \begin{array}[]{c}\displaystyle\frac{1-d-i}{2},\frac{2-d-i}{2}\\b+\frac{1}{2}\end{array};\zeta\right)}}\end{equation*}}
	  converges when $| \zeta | < 1$, and we have the following closed formula:

\begin{equation}
    G (a, b, d ; \zeta)=
    \frac{{2^{1 - d}}{}\sqrt{\pi}{}{\Gamma\left( d \right)}}
    {{\Gamma\left( \frac{a+d}{2} \right)}{}{\Gamma\left( \frac{1-a+d}{2} \right)}}{}
         {{}_2F_1\left( \begin{array}[]{c}\displaystyle1-\frac{a+d}{2},\frac{1+a-d}{2}\\b+\frac{1}{2}\end{array};\zeta\right)}.
    \label{eqn:iF}
\end{equation}
 	\end{lemma}

	Postponing the verification of Lemmas
	\ref{lem4} and \ref{lem:Fisum}, we first
	show Proposition \ref{prop:2}.

	{\noindent{\textit{\textbf{Proof of Proposition \ref{prop:2}.}\ }}}
By the change of variables $s = 1 - (1 {-} x) t$, the interval $0 \leqslant t
	  \leqslant 1$ is transformed onto $(- 1 \leqslant) {}  x \leqslant s \leqslant
	  1$, and thus Euler's integral representation of the hypergeometric function
	  $_2 F_1$ shows
	  
\begin{equation*}
    {\displaystyle\int_{-1}^{1}{\left( s-x \right)_+^{2c-1}}{}{\left(1-s^2\right)^{a-1}}{d}s}
    =
    {2^{a-1}}{}
    {B\left(2c,a\right)}{}
    {\left(1-x^2\right)^{2c+a}}{}
    {{}_2F_1\left( \begin{array}[]{c}\displaystyle1-a,2c\\2c+a\end{array};\frac{1-x}{2}\right)}
        .
\end{equation*}
 	  
	  Therefore the left-hand side of $(\ref{eqn:stz})$ equals
	  \begin{equation*}
		  2^{a-1} B (2 c, a) \displaystyle\int_{- 1}^1 (1 {-} tx)^{2 c + a - 1}{}_2 F_1 \left(
	    \begin{array}{c}
	      1 - a, a\\
	      2 c + a
	    \end{array} ; \frac{1 {-} tx}{2} \right) (1 - t^2)^{b - 1} d t.   
	  \end{equation*}
	  Fix $\varepsilon > 0$. Assume $| x | \leqslant 1 - 2 \varepsilon$. Then
	  $\left| \frac{1 {-} tx}{2} \right| \leqslant 1 - \varepsilon$. Expanding the
	  hypergeometric function as a uniformly convergent power series of
	  $\frac{1}{2} (1 {-} tx)$, we can rewrite the integral in the right-hand side
	  as
	  \begin{eqnarray}
	    & \displaystyle\sum_{i = 0}^{\infty} \frac{(1 - a)_i (a)_i}{2^i i! (2 c + a)_i}
	    \displaystyle\int_{- 1}^1 (1 {-} t x)^{2 c + a - 1 + i} (1 - t^2)^{b - 1} d t. & 
	    \nonumber
	  \end{eqnarray}
	  Owing to Lemma \ref{lem4}, this is equal to
	  \begin{eqnarray}
	    & B \left( \frac{1}{2}, b \right) \displaystyle\sum_{i = 0}^{\infty} \frac{(1 - a)_i
	    (a)_i}{2^i i! (2 c + a)_i}{}_2 F_1 \left( \begin{array}{c}
	      \displaystyle\frac{1 - 2 c - a - i}{2}, \frac{2 - 2 c - a - i}{2}\\
	      b + \frac{1}{2}
	    \end{array} ; x^2 \right) . &  \nonumber
	  \end{eqnarray}
	  Now (\ref{eqn:stz}) follows from Lemma \ref{lem:Fisum} with $\zeta = x^2$
	  and $d = 2 c + a$ and from the
	  duplication formula of the Gamma function $\Gamma (2 c) = \pi^{-
	  \frac{1}{2}} 2^{2 c - 1} \Gamma (c) \Gamma \left( c + \frac{1}{2} \right)$ as far as $-1<x<1$.
	  Finally, by the continuity, \eqref{eqn:stz} holds for $x=\pm1$ under the assumption on the parameters $a,b,c\in\mathbb{C}$.
{\hspace*{\fill}$\qed$\medskip}

{\noindent{\textit{\textbf{Proof of Lemma \ref{lem4}.}\ }}}
	  By Euler's integral representation of $_2 F_1$ again,
	  the left-hand side of \eqref{eqn:EulerRep} amounts to
	  \begin{eqnarray}
	    & \displaystyle2^{2 b - 1} (1{-} 
	    x)^{a - 1} B (b, b)_2 F_1 \left( \begin{array}{c}
	      1 - a, b\\
	      2 b
	    \end{array} ; \frac{{-} 2 x}{1 {-} x} \right) .  \label{eqn:quad} & 
	  \end{eqnarray}
	  Applying the quadratic transformation of $_2 F_1$ (cf. {\cite[Thm.
	  3.13]{andrews2000special}}):
	  \begin{eqnarray}
	    & \;_2 F_1 \left( \begin{array}{c}
	      \displaystyle1 - a, b\\
	      2 b
	    \end{array} ; u \right) = \displaystyle\left( 1 - \frac{u}{2} \right)^{a - 1}{}_2 F_1
	    \left( \begin{array}{c}
	      \displaystyle\frac{1 - a}{2}, \frac{2 - a}{2}\\
	      b + \frac{1}{2}
	    \end{array} ; \displaystyle\left( \frac{u}{2 - u} \right)^2 \right), &  \nonumber
	  \end{eqnarray}
	  with $u = \frac{{-} 2 x}{1 {-} x}$, we get the desired result after a small
	  computation using the duplication formula of the Gamma function.
{\hspace*{\fill}$\qed$\medskip}

{\noindent{\textit{\textbf{Proof of Lemma \ref{lem:Fisum}.}\ }}}
	  We list some elementary identities for the Pochhammer symbol $(y)_j =
	  \frac{\Gamma (y + j)}{\Gamma (y)}$:
	  \begin{eqnarray}
	    & 
	    (y)_i\Gamma(1-y-i)&=\;(-1)^i\Gamma(1-y)
	    ,  \label{eqn:p1}\\
	    & \left( \frac{y}{2} \right)_j \left( \frac{1 + y}{2} \right)_j & = \;
	    2^{- 2 j} (y)_{2 j},  \label{eqn:p2}\\
	    & (y)_i (1 - y)_{2 j} & = \; (1 - y - i)_{2 j} (y - 2 j)_i . 
	    \label{eqn:p3}
	  \end{eqnarray}
	  To prove the equation \eqref{eqn:iF}, we first show the following expansion:
	  \begin{eqnarray}
	    & G (a, b, d ; \zeta) = \displaystyle\sum_{j = 0}^{\infty} \frac{(1 - d)_{2 j}
	    \zeta^j}{2^{2 j} j! \left( b + \frac{1}{2} \right)_j}{}_2 F_1 \left(
	    \begin{array}{c}
	      a, 1 - a\\
	      d - 2 j
	    \end{array} ; \frac{1}{2} \right) . &  \label{eqn:Fijsum}
	  \end{eqnarray}
	  Indeed, by expanding the hypergeometric function as a power series and by
	  using $(\ref{eqn:p2})$ with $y = 1 - d - i$, we have
	  \begin{eqnarray}
	    & G (a, b, d ; \zeta) = \displaystyle\sum_{i = 0}^{\infty} \displaystyle\sum_{j = 0}^{\infty}
	    \frac{(a)_i (1 - a)_i}{2^{i + 2 j} i!j! (d)_i}  \frac{(1 - d - i)_{2
	    j}}{\left( b + \frac{1}{2} \right)_j} \zeta^j, &  \nonumber
	  \end{eqnarray}
	  which is equal to the right-hand side of $(\ref{eqn:Fijsum})$ by
	  $(\ref{eqn:p3})$ with $y = d$.
	  
	  As $\;_2 F_1 \left( \begin{array}{c}
	    a, 1 - a\\
	    c
	  \end{array} ; \displaystyle\frac{1}{2} \right) =\displaystyle \frac{2^{1 - c} \sqrt{\pi} \Gamma
	  (c)}{\Gamma \left( \frac{a + c}{2} \right) \Gamma \left( \frac{c - a + 1}{2}
	  \right)}$ (see {\cite[Thm. 5.4]{andrews2000special}} for instance), we can
	  continue as
	  \begin{eqnarray}
	    & \begin{array}{ll}
	      (\ref{eqn:Fijsum}) & = \displaystyle\sum_{j = 0}^{\infty} \frac{(1 - d)_{2 j}
	      \zeta^j}{2^{2 j} j! \left( b + \frac{1}{2} \right)_j} \cdot \frac{2^{1 -
	      d + 2 j} \sqrt{\pi} \Gamma (d - 2 j)}{\Gamma \left( \frac{a + d}{2} - j
	      \right) \Gamma \left( \frac{1 - a + d}{2} - j \right)}\\
	      &\displaystyle = \frac{2^{1 - d} \sqrt{\pi} \Gamma (d)}{\Gamma \left( \frac{a +
	      d}{2} \right) \Gamma \left( \frac{1 - a + d}{2} \right)} \displaystyle\sum_{j =
	      0}^{\infty} \frac{\left( 1 - \frac{a + d}{2} \right)_j \left( \frac{1 +
	      a - d}{2} \right)_j}{j! \left( b + \frac{1}{2} \right)_j} \zeta^j,
	    \end{array} &  \nonumber
	  \end{eqnarray}
	  where we have used $(\ref{eqn:p1})$ with 
	  $(y,i)=(1-d,2j),(1-\frac{1}{2}(a+d),j)$,
	  and $\left( \frac{1}{2}(1+a-d),j \right)$
	  in the second equality. Hence Lemma
	  \ref{lem:Fisum} is proven.
{\hspace*{\fill}$\qed$\medskip}
	\section{Sobolev-type estimate for Gegenbauer expansion}\label{sec:Sobolev}
In this section we formulate a Sobolev-type estimate for Gegenbauer expansion, by which
Theorem \ref{thm:1-1} follows readily from the special value of the integral formula (Theorem \ref{main-thm}),
as we shall see in Section \ref{sec:pfThm}.

We begin with a basic {setup}.
If $\lambda>-\frac{1}{2}$ and $\lambda\neq0$, then the Gegenbauer polynomials $\left\{ C_n^\lambda \right\}_{n\in\mathbb{N}}$
form an orthogonal basis in the Hilbert space $L_\lambda^2:=L^2\left( (-1,1),(1-x^2)^{\lambda-\frac{1}{2}}dx \right)$
with the norm\begin{equation}
	\label{eqn:vnlmd}
	v_n^\lambda:=\|C_n^\lambda\|
	^2_{L^2_\lambda}=\frac{2^{1-2\lambda}\pi\Gamma(n+2\lambda)}{n!(n+\lambda)\Gamma(\lambda)^2}.
\end{equation}
This means that any $f\in L_\lambda^2$, has an $L^2$-expansion\begin{equation}
	\label{eqn:aGegen}
	f(x)=\displaystyle\sum_{n=0}^\infty a_n(f)C_n^\lambda(x),
\end{equation}where $a_n(f)\in\mathbb{C}$ is given by
\begin{equation*}
	a_n(f)=\frac{1}{v_n^\lambda}\displaystyle\int_{-1}^1f(x)C_n^\lambda(x)(1-x^2)^{\lambda-\frac{1}{2}}dx.
\end{equation*}
\begin{proposition}
	\label{prop:1718105}(Sobolev-type inequality for Gegenbauer expansion)
	Suppose $\lambda>0$. Then there exists $D_\lambda>0$ with the following property:
	let $N:= [\lambda]+2$, the integer satisfying $\lambda+1<N\le \lambda+2$.
	Then
	\begin{equation}
		\|f\|_{L^\infty(-1,1)}\le 
		D_\lambda
		\left( \|f\|_{L^2_\lambda}+
		\|f^{(N)}\|_{L^2_{\lambda}} \right)	
		\label{eqn:Sobolev}
	\end{equation}
	for any $f\in L_\lambda^2$ such that the $N$-th derivative $f^{(N)}$ belongs to $L^2_{\lambda}$.
	Moreover, the Gegenbauer expansion
	\eqref{eqn:aGegen} converges absolutely and uniformly in $[-1,1]$ for any such $f$.
\end{proposition}
\begin{remark}
	Note that for $\lambda=\frac{1}{2}$, $L^2_{\lambda}=L^2(0,1)$ and \eqref{eqn:Sobolev} 
    follows from the classical
	Sobolev inequality.
	
\end{remark}
\begin{remark}
	We note that there is a continuous embedding\begin{equation*}
		L^2_\lambda\hookrightarrow L^2_{\lambda+a}\quad\mbox{for any {$a>0$}.}
	\end{equation*}
	As the proof shows, we may strengthen Proposition \ref{prop:1718105} by replacing \eqref{eqn:Sobolev} with
	\begin{equation*}
		\|f\|_
		{L^\infty(-1,1)}\le D_{\lambda}
		\left( \|f\|_
		{L^2_\lambda}+\|f^{(N)}\|_
		{L^2_{\lambda+N}} \right)
	\end{equation*}
	and $f^{(N)}\in L^2_\lambda$ with $f^{(N)}\in {{L^2_{\lambda+N}}}$.
\end{remark}
The rest of this section is devoted to the proof of Proposition \ref{prop:1718105}. We begin with the following.
\begin{lemma}
	Suppose $\mathbb{N}\ni N>\lambda>0$.
	Then there is a constant $d_{\lambda,N}>0$ such that \begin{equation*}
        \|C_n^\lambda\|_{L^\infty(-1,1)}\le d_{\lambda,N}\|C^{\lambda+N}_{n-N}\|_{L^2_{\lambda+N}}n^{\lambda-N}\mbox{ for all $n\ge N$.}
	\end{equation*}
	\label{lem:1718109}
\end{lemma}
\begin{proof}
	We recall from \cite[(6.4.11)]{andrews2000special} that
	\begin{equation}
		\label{eqn:Cnbdd}
		{\left|C_n^\lambda(x)\right|}\le C_n^\lambda(1)=\frac{\Gamma(n+2\lambda)}{n!\Gamma(2\lambda)}\mbox{ for all $-1\le x\le 1$.}
	\end{equation}
	By \eqref{eqn:vnlmd}
	\begin{equation*}
		\frac{C^\lambda_n(1)^2}{v_{n-N}^{\lambda+N}}
		=\frac{\Gamma(\lambda+N)^2}{\Gamma(2\lambda)^2\,2^{1-2\lambda-2N}\pi}\cdot
		\frac{\Gamma(n+2\lambda)^2(n-N)!(n+\lambda)}{\left( n! \right)^2\Gamma(n+2\lambda+N)}.
	\end{equation*}
	The first term depends only on $\lambda$ and $N$,
	and the second term has the following asymptotics: $n^{2\lambda-2N}$
	as $n$ tends to $\infty$ because\begin{equation*}
		\frac{\Gamma(n+a)}{\Gamma(n+b)}\sim n^{a-b}\quad\mbox{as $n\to\infty$.}
	\end{equation*}
	Now Lemma \ref{lem:1718109} follows from \eqref{eqn:Cnbdd}.
\end{proof}

We are ready to complete the proof of Proposition \ref{prop:1718105}.

{\noindent{\textit{\textbf{Proof of Proposition \ref{prop:1718105}.}\ }}}
	Let $N:= [\lambda]+2$ be as in Proposition \ref{prop:1718105}.
	Iterating the differential formula\begin{equation*}
		\frac{d}{dx}C_n^\lambda(x)=2\lambda C^{\lambda+1}_{n-1}(x),
	\end{equation*}we {get} the following $L^2$-expansion:\begin{equation*}
		f^{(N)}(x)=2^N(\lambda)_N\displaystyle\sum_{n=N}^\infty a_n(f)C_{n-N}^{\lambda+N}(x).
	\end{equation*}Thus, for all $n\ge N$, we have\begin{equation*}
		{\left|a_n(f)\right|}\le\frac{1}{2^N(\lambda)_N}
		\frac{\|f^{(N)}\|_{L^2_{\lambda+N}}}
		{\|C^{\lambda+N}_{n-N}\|_
		{L^2_{\lambda+N}}}.
	\end{equation*}
	By Lemma \ref{lem:1718109}
	\begin{equation*}
		{\left|a_n(f)\right|}\|C^\lambda_n\|_{L^\infty(-1,1)}\le\frac{d_{\lambda,N}}
		{2^N(\lambda)_N}\|f^{(N)}\|_{L^2_{\lambda+N}}n^{\lambda-N}.
	\end{equation*}Therefore
	the right-hand side of
	\eqref{eqn:aGegen} converges uniformly in $[-1,1]$ because $\lambda-N<-1$.

	For $0\le n<N$, we use ${\left|a_n(f)\right|}\sqrt{v_n^\lambda}\le\|f\|_{L_\lambda^2}$ to conclude
	\begin{equation*}
		\left(\sum_{n=0}^{N-1}+\sum_{n=N}^\infty  \right) a_n(f)\|C_n^\lambda\|_{L^\infty(-1,1)}\le D_\lambda\left( \|f\|_{L^2_\lambda}+\|f^{(N)}\|_{L^2_{\lambda+N}} \right),
	\end{equation*}where we set\begin{equation*}
		D_\lambda:=\max\left( \frac{d_{\lambda,N}}{2^N(\lambda)_N}\displaystyle\sum_{n=N}^{\infty}n^{\lambda-N},\left\{ \frac{\|C_n^\lambda\|_{L^\infty(-1,1)}}{\sqrt{v_n^\lambda}} \right\}_{n=0,\cdots,N-1} \right).
	\end{equation*}Hence Proposition \ref{prop:1718105} is proved.
{\hspace*{\fill}$\qed$\medskip}
 	\section{Proof of Theorem \ref{thm:1-1}}\label{sec:pfThm}
We obtain from Theorem \ref{main-thm} the following.
\begin{proposition}\label{prop:4:6d1}
	With the same assumption as in Theorem 
	\ref{main-thm}, we have
	\begin{equation*}
	\begin{array}{ll}
     \displaystyle\int_{- 1}^1 \displaystyle\int_{- 1}^1 {\left( s {-} xt \right)}_-^{2 \nu} u_{\ell}^{\lambda} (s)
     u_m^{\mu} (t) d s d t   &=(-1)^{\ell+m}{{B^{\lambda,\mu,\nu}_{\ell,m}\left( x \right)}},\nonumber\\
     \displaystyle\int_{- 1}^1 \displaystyle\int_{- 1}^1 | s {-} xt |^{2 \nu} u_{\ell}^{\lambda} (s)
    u_m^{\mu} (t) d s d t   
    &=\left( 1+(-1)^{\ell+m}
    \right){{B^{\lambda,\mu,\nu}_{\ell,m}\left( x \right)}},\nonumber
\\
	\displaystyle\int_{- 1}^1 \displaystyle\int_{- 1}^1 | s {-} xt |^{2 \nu}\operatorname{sgn}\left( s{-} xt \right) u_{\ell}^{\lambda} (s)
	u_m^{\mu} (t) d s d t&=\left( 1-(-1)^{\ell+m} 
	\right){{B^{\lambda,\mu,\nu}_{\ell,m}\left( x \right)}}.\nonumber
	\end{array}
	\end{equation*}
\end{proposition}
\begin{proof}
    Taking into account
    that $y^\alpha_-=(-y)_+^\alpha$
    and $u^\lambda_\ell(-s)=(-1)^\ell u^\lambda_\ell(s)$,
    one derives the first integral formula from
    Theorem \ref{main-thm}. In turn, the second
    and the third ones hold by the change of the basis
    $\left\{ {\left|y\right|}^\alpha,
    {\left|y\right|}^\alpha{\operatorname{sgn}}{\left(y\right)}\right\}$
    to $\left\{ y^\alpha_+,y^\alpha_- \right\}$,
    see \eqref{eqn:Riesz}.
\end{proof}
 \begin{proposition}
  \label{cor:1}
  Let $\varepsilon\in\left\{ 0,1 \right\}$ and $\ell,m\in\mathbb{N}$. Suppose ${\operatorname{Re}\lambda},{\operatorname{Re}\mu}>-\frac{1}{2}$ and ${\operatorname{Re}\nu}>0$.
  \begin{alignat}{1}
    &\int_{- 1}^1 \displaystyle\int_{- 1}^1 | s {-} t |^{2 \nu}{\operatorname{sgn}}^\varepsilon(s{-} t) (1 - s^2)^{\lambda -
    \frac{1}{2}} (1 - t^2)^{\mu - \frac{1}{2}} C_{\ell}^{\lambda} (s)
    C_m^{\mu} (t) d s d t   \nonumber\\
    &=\frac{1}{2}\left( 1+(-1)^{\ell+m+\varepsilon} \right){{b^{\lambda,\mu,\nu}_{\ell,m}}} v_\ell^\lambda v^\mu_m.
	\label{eqn:cor:1}
  \end{alignat}
\end{proposition}
  \begin{proof}
	  Since $_2F_1\left( \begin{array}[]{c}
		  a,b\\c
	  \end{array};
	  1\right)=
	  \displaystyle
	  \frac{\Gamma(c-a-b)\Gamma(c)}{\Gamma(c-a)\Gamma(c-b)}$ if ${\operatorname{Re}c}>\operatorname{Re}\,(a+b)$, the
	  left-hand side of \eqref{eqn:cor:1} amounts to
  \begin{equation*}
	  \displaystyle
	  \frac{{\left((-1)^m+(-1)^{\ell+\varepsilon}\right)}\pi^{\frac{1}{2}}
{
     (2 \lambda)_{\ell} (2 \mu)_m \Gamma \left( \lambda +
    \frac{1}{2} \right) \Gamma \left( \mu + \frac{1}{2} \right) \Gamma \left(
    \nu + \frac{1}{2} \right)\Gamma(\nu+1) \Gamma (\lambda + \mu + 2 \nu + 1)}
  }{2
{
	\Gamma(1+ \nu-\frac{\ell+m}{2})
	    \ell !m!
    \Gamma \left( \lambda + \nu + \frac{\ell - m}{2} + 1 \right) \Gamma \left(
    \mu + \nu - \frac{\ell - m}{2} + 1 \right) \Gamma \left( \lambda + \mu +
    \nu + \frac{\ell + m}{2} + 1 \right)}
  }
  \end{equation*}
  from the second and third {formul\ae} of Proposition \ref{prop:4:6d1} with $x=1$. By the definition
  \eqref{eqn:alm} of ${{b^{\lambda,\mu,\nu}_{\ell,m}}}$
 and the formula \eqref{eqn:vnlmd} of $v^\lambda_\ell$,
the proposition follows.
  \end{proof}
We are ready to complete the proof of Theorem \ref{thm:1-1}.\\
{\noindent{\textit{\textbf{Proof of Theorem \ref{thm:1-1}.}\ }}}
Owing to Proposition \ref{prop:1718105}, we can deduce
Theorem \ref{thm:1-1}
from Proposition \ref{cor:1}
under the assumption on $(\lambda,\mu,\nu)$
because for any $m,n\in\mathbb{N}$ with $m\le\lambda+2$ and $n\le\mu+2$, we have\begin{equation*}
	\frac{\partial^{m+n}}{\partial s^m\partial t^n}{\left|s{-} t\right|}^{2\nu}\in L^2\left( (-1,1)^2,(1-s^2)^{\lambda+m}(1-t^2)^{\mu+n}dsdt \right).
\end{equation*}
Hence Theorem \ref{thm:1-1} is proved.
{\hspace*{\fill}$\qed$\medskip}

\begin{remark}
\label{rem:alt-Thm11}
	  Kobayashi--Mano obtained an analogous formula
 to \eqref{eqn:cor:1} in \cite[Lem.~7.9.1]{kobayashi2011schrodinger},
{}from which Proposition \ref{cor:1} follows 
 by the change of basis \eqref{eqn:Riesz}
 and thus we could give an alternative proof of Theorem \ref{thm:1-1}.  
Our proof of \eqref{eqn:cor:1} is different from \cite[Chap.~7]{kobayashi2011schrodinger}, 
 where they showed the following integral formula
 \cite[(7.4.11)]{kobayashi2011schrodinger}
 as a first step:
 for $\operatorname{Re}\lambda >-1$, 
 $\operatorname{Re}\nu >- \frac 12$ and $|x|< 1$, 
\begin{equation}
\label{eqn:KMP}
 x_-^{\lambda} \ast h_k^{\nu}(x)
 =
 q_k(\lambda,\nu)(1-x^2)^{\frac 1 2(\lambda+\nu+\frac 1 2)}P_{\nu+k-\frac 1 2}^{-(\lambda+\nu+\frac 1 2)}(-x).  
\end{equation}
Here $h_k^{\nu}(x):=(1-x^2)^{\nu-\frac 1 2}C_k^{\nu}(x)$
 for $|x|<1$;
 $=0$ otherwise, 
$P_{\alpha}^{\beta}(x)$ denotes the associated Legendre function, 
 and the coefficient $q_k(\lambda,\nu)$ is given explicitly
 by the Gamma functions.  
The integral formula \eqref{eqn:KMP} immediately implies
 closed formul{\ae} for 
\[
  x_+^{\lambda} \ast h_k^{\nu} \quad \text{and}\quad
  (x \pm i 0)^{\lambda} \ast h_k^{\nu}
\]
because $C_k^{\nu}(-x) =(-1)^k C_k^{\nu}(x)$.  
With the notation as in \eqref{eqn:vnlmd}, 
 the integral formula \eqref{eqn:KMP} implies an expansion
\begin{equation}
\label{eqn:KMPC}
  (x-y)_-^{\lambda}
  =
  \sum_{k=0}^{\infty} \frac{q_k(\lambda,\nu)}{v_k^{\lambda}} 
  (1-x^2)^{\frac 1 2(\lambda+\nu+\frac 1 2)}
  P_{\nu+k-\frac  1 2}^{-(\lambda+\nu+\frac 12)}(-x) C_k^{\nu}(y)
\end{equation}
for any ${\mathbb{R}} \ni \nu > -\frac 1 2$ 
 with $\nu \ne 0$, 
 and similarly for $(x-y)_+^{\lambda}$ and $(x-y\pm i 0)^{\lambda}$.  
Kobayashi--Mano's work \cite{kobayashi2011schrodinger} appeared 
 in arXiv:0712.1769.  
Afterwards, 
 Cohl \cite{cohl2013} and Szmytkowski \cite{szmytkowski2011} obtained
 similar results to \eqref{eqn:KMP} and \eqref{eqn:KMPC}, 
 but not the double Gegenbauer expansion
 as in \eqref{eqn:cor:1}.  
To be more precise, 
 Szmytkowski \cite[(2.5), (2.7)]{szmytkowski2011} rediscovered
 the same formula with \eqref{eqn:KMP} by using from Cohl \cite[Thm.~2.1]{cohl2013}.  
We note that \cite[Thm.~2.1]{cohl2013}
 also follows from Kobayashi--Mano's formula \cite[Lem.~7.9.1]{kobayashi2011schrodinger}
 by  change of basis \eqref{eqn:Riesz2} and analytic continuation.  
\end{remark}

 	\section{Proof of Corollary \ref{cor:CCint}}\label{sec:pfOfCol}
It is sufficient to prove the following.
\begin{lemma}
  \label{lem:VXGPjH8UqL-35d3}
  Suppose $\lambda,\mu,\nu,\beta\in\mathbb{C}$ and $\ell,m\in\mathbb{N}$
  satisfy
  ${\operatorname{Re}\beta} > - 1, \operatorname{Re}\,(\mu + m )> - 1$, and $\operatorname{Re}\,(\lambda+\mu  + 2 \nu+ \beta
  + 2) > 0$. Then we have
  \begin{alignat*}{1}
       &\int_0^1 x^{2 \mu + m + 1} (1 - x^2)^\beta B^{\lambda, \mu, \nu}_{\ell, m}
       (x) d x \\
       &= \frac{(- 1)^m 2^{- 2 \nu-2} \pi^2 \Gamma (2 \nu + 1) \Gamma (\beta + 1)
       \Gamma (\lambda+\mu + 2 \nu +\beta + 2)}{\Gamma \left( \nu - \frac{\ell
       + m}{2}+1 \right) \Gamma \left( \lambda + \nu + \frac{\ell - m}{2} + 1
       \right) \Gamma \left( \mu + \nu + \beta + \frac{m - \ell}{2} + 2 \right)
       \Gamma \left( \lambda+\mu + \nu + \beta + \frac{m + \ell}{2} + 2 \right)}.
  \end{alignat*}
\end{lemma}
     To prove Lemma \ref{lem:VXGPjH8UqL-35d3}, we use \cite[20.2~(4)]{MR0065685}:
     \begin{equation}
         \displaystyle\int_0^1 y^{\gamma - 1} (1 - y)^{\rho - 1} \,_2 F_1 \left(
         \begin{array}{c}
           \alpha, \beta\\
           \gamma
         \end{array} ; y \right) d y =\displaystyle \frac{\Gamma (\gamma) \Gamma
         (\rho) \Gamma (\gamma + \rho - \alpha - \beta)}{\Gamma (\gamma + \rho
         - \alpha) \Gamma (\gamma + \rho - \beta)},\\
	     \label{eqn:Fint}
     \end{equation}
     if ${\operatorname{Re}\gamma} > 0, {\operatorname{Re}\rho} > 0,
         \operatorname{Re}\,(\gamma + \rho - \alpha - \beta) > 0$.
\begin{proof}
	By the change of variables $x=y^2$, the formula \eqref{eqn:Fint} shows
  \[ 
	 \begin{array}{c}
       \displaystyle \int_0^1 x^{2 \mu + 2 m + 1} (1 - x^2)^\beta \,_2 F_1  \left( 
       \begin{array}{c}
         -\nu+\frac{\ell + m}{2} , -\lambda-\nu-\frac{ \ell-m}{2} \\
         \mu + m + 1
       \end{array} ; x^2  \right) d x\\
       =\displaystyle \frac{\Gamma (\mu + m
       + 1) \Gamma (\beta + 1)
       \Gamma (\lambda+\mu+ 2\nu + \beta + 2  )}
       {2 \Gamma \left( \mu + \nu+ \beta- \frac{\ell-m}{2}  + 2 
       \right) \Gamma \left( \lambda+\mu +\nu+ \beta+ \frac{\ell+m}{2}  + 2 
       \right)} .
     \end{array} \]
     Now the lemma follows from Theorem \ref{main-thm}.
\end{proof}
 	\section{{{Special values and Selberg-type integrals}}}\label{sec:4}
	In this section, we examine the relationship between
	Theorem \ref{main-thm} and some known integral formul{\ae} by Selberg, Dotsenko,
	Fateev, Tarasov, Varchenko and Warnaar among others when the parameters take
	special values. The hierarchy of the formul{\ae} treated here is summarized in
	Figure \ref{fig:intdep}.

	For this, we limit ourselves to the special case of Theorem \ref{main-thm}
	with $(\ell, m, x) = (0, 0, {}  1)$, or equivalently, of Proposition \ref{cor:1}
	with $(\ell, m) = (0, 0)$:
	
  \begin{alignat}{1}
        &{\displaystyle\int_{-1}^{1}\int_{-1}^{1}{\left|s-t\right|^{2{}\nu}}{}
        {\left(1-s^2\right)^{\lambda-\frac{1}{2}}}{}
        {\left(1-t^2\right)^{\mu-\frac{1}{2}}}{d}s{d}t}
	\nonumber
	  \\&=
      \frac{{\pi^{\frac{1}{2}}}
      {}{\Gamma\left( \lambda + \frac{1}{2} \right)}
      {}{\Gamma\left( \mu+\frac{1}{2} \right)}
      {}{\Gamma\left( \nu+\frac{1}{2} \right)}
      {}{\Gamma\left( \lambda+\mu+2{}\nu+1 \right)}}
      {{\Gamma\left( \lambda+\nu+1 \right)}{}{\Gamma\left( \mu+\nu+1 \right)}{}{\Gamma\left( \lambda+\mu+\nu+1 \right)}}
      .
\label{eqn:lm0}
	\end{alignat}
 	\begin{example}
  \label{ex:1}(Selberg integral {\cite{Selberg:411367}}) The Selberg integral
  \begin{alignat}{1}
     &\int_0^1 \ldots \displaystyle\int_0^1 \displaystyle\prod_{i = 1}^n t_i^{\alpha - 1} (1 -
    t_i)^{\beta - 1} \left| \displaystyle\prod_{1 \leqslant i < j \leqslant n} (t_i - t_j)
    \right|^{2 \nu} d t_1 \cdots d t_n    \nonumber\\
     &= \displaystyle\prod_{j = 1}^n \frac{\Gamma (\alpha + (j - 1) \nu) \Gamma (\beta +
    (j - 1) \nu) \Gamma (1 + j \nu)}{\Gamma (\alpha + \beta + (n + j - 2) \nu)
    \Gamma (1 + \nu)}
	 \label{eqn:selberg} 
  \end{alignat}
  is a generalization of the Euler beta integral. 
  The special case of Theorem
  \ref{main-thm} with $(\ell,m,x,\mu)=(0,0,{}  1,\lambda)$, namely, 
  \eqref{eqn:lm0}
  with $\lambda=\mu$
  reduces to the
  special case of {\eqref{eqn:selberg}}
  with $(n, \alpha, \beta) = \left( 2, \lambda +
  \frac{1}{2}, \lambda + \frac{1}{2} \right)$, namely,
  
  \begin{equation}
      {\displaystyle\int_{-1}^{1}\int_{-1}^{1}
          {\left(1-s^2\right)^{\lambda-\frac{1}{2}}}{}
          {\left(1-t^2\right)^{\lambda-\frac{1}{2}}}{}
          {\left|s-t\right|^{2\nu}}{d}s{d}t}
    = 
    \frac{{2^{4{}\lambda+2{}\nu}}{}
    {{\Gamma\left( \lambda + \frac{1}{2} \right)}^{2}}}
    {{\Gamma\left( 2{}\lambda + 1 + \nu \right)}}
     \cdot
     \frac{{{\Gamma\left( \lambda + \nu + \frac{1}{2} \right)}^{2}}
     {\Gamma\left( 1 + 2 \nu \right)}}
     {{\Gamma\left( 2{} \lambda + 2{} \nu + 1 \right)}{}{\Gamma\left( 1 +\nu \right)}}
    ,
    \label{eqn:spec_selberg}
  \end{equation}
   after a change of variables $(t_1, t_2) = \left( \frac{1 + s}{2}, \frac{1 +
  t}{2} \right)$. 
\end{example}
 	\begin{example}
  \label{ex:2}(Warnaar integral) 
  The special case of Theorem \ref{main-thm} with $(\ell,
  m, x, \nu) = \left( 0, 0, {}  1, - \frac{\lambda + \mu}{2} \right)$, namely,
  \eqref{eqn:lm0} with
  $\lambda+\mu+2\nu=0$ reduces to a special case of Warnaar's integral formula
  {\cite[(1.4)]{warnaar2010sl3}}
  with $(k_1, k_2, \alpha_1, \beta_1,
  \alpha_2 , \beta_2, \gamma) = \left( 1, 1, \lambda + \frac{1}{2},
  \lambda + \frac{1}{2}, \mu + \frac{1}{2}, \mu + \frac{1}{2} , \lambda +
  \mu \right)$, namely,
  
  \begin{alignat}{2}
&{\left({}\int\int_{{0}\le{s}<{t}\le{1}}+{\frac{{\cos\left(\pi{}\lambda\right)}}{{\cos\left(\pi{}\mu\right)}}}\int\int_{{0}\le{t}<{s}\le{1}}\right){{t^{\lambda-\frac{1}{2}}}{}{\left(1-t\right)^{\lambda-\frac{1}{2}}}{}{s^{\mu-\frac{1}{2}}}{}{\left(1-s\right)^{\mu-\frac{1}{2}}}}{\left|{s}-{t}\right|}^{-\lambda-\mu}{}d{s}d{t}}
\nonumber\\
    &=
    \frac{{\Gamma\left( \lambda + \frac{1}{2} \right)}{}
    {\Gamma\left( \frac{1}{2}-\mu \right)}{}{{\Gamma\left( \mu+\frac{1}{2} \right)}^{2}}}
    {{\Gamma\left( \lambda+1-\mu \right)}{}{\Gamma\left( \mu+1-\lambda \right)}{}{\Gamma\left( \lambda+\mu+1 \right)}}
      .\label{eqn:spec_warnaar}
  \end{alignat}
 \end{example}
 	\begin{example}
  \label{ex:3}
  ($\mathfrak{s}\mathfrak{l}_3$ Selberg integral of Tarasov and
  Varchenko)
  The special case of Theorem \ref{main-thm} with $(\ell,
  m, x, \mu) = \left( 0, 0, {} 1, \frac{1}{2} \right)$, namely, \eqref{eqn:lm0} with $\mu=\frac{1}{2}$
  reduces to a special case of Tarasov--Varchenko's
  integral formula {\cite[(3.4)]{tarasov2003selberg}} with
  $(k_1, k_2, \alpha, \beta_1, \beta_2, \gamma) = \left( 1, 1,
  \lambda + \frac{1}{2}, \lambda + \frac{1}{2}, 1, - 2 \nu \right)$, namely,
  
  \begin{equation}
      {\displaystyle\int_{-1}^{1}\int_{-1}^{1}{\left(1-s^2\right)^{\lambda-\frac{1}{2}}}{}
      {\left( t-s \right)_+^{2\nu}}{d}s{d}t}
    =
    \frac{{2^{2{}\lambda+2{}\nu+1}}{}
    {\Gamma\left( \lambda+\frac{1}{2} \right)}{}{\Gamma\left( \frac{3}{2}+\lambda+2{}\nu \right)}}
    {(1 + 2{}\nu){} {\Gamma\left( 2+2{}\lambda+2{}\nu \right)}}
.\label{eqn:spec_tv}
  \end{equation}
 \end{example}
 	\begin{example}
  \label{ex:4}(Dotsenko--Fateev integral) 
  The special case of Theorem \ref{main-thm} with $(\ell,
  m, x, \nu) = (0, 0, {} 1, - 1)$, namely, \eqref{eqn:lm0} with $\nu=-1$ reduces to a special case of Dotsenko--Fateev's
 integral formula
 {\cite[$(A1)=(A35)$]{dotsenko1985four}} with $(n, m, \alpha, \beta, \rho) =
  \left( 1, 1, \mu - \frac{1}{2}, \mu - \frac{1}{2}, - \frac{\mu -
  \frac{1}{2}}{\lambda - \frac{1}{2}} \right)$, namely,
  
\begin{equation}
{\displaystyle\int_{-1}^{1}\int_{-1}^{1}{\left(1-s^2\right)^{\lambda-\frac{1}{2}}}{}{\left(1-t^2\right)^{\mu-\frac{1}{2}}}{}{\left|s-t\right|^{-2}}{d}s{d}t}=
    \frac{ {2^{2\lambda+2\mu-1}} {}
    {{\Gamma\left( \lambda+\frac{1}{2} \right)}^{2}}{}
    {{\Gamma\left( \mu+\frac{1}{2} \right)}^{2}}}
    {(1-\lambda - \mu ){} {\Gamma\left( 2{}\lambda \right)}{}{\Gamma\left( 2{}\mu \right)}} 
    .
    \label{eqn:spec_df}
  \end{equation}
 \end{example}
 
	The hierarchy of the integral formul{\ae} in Examples \ref{ex:1}--\ref{ex:4}
	and Theorem \ref{main-thm} is summarized as follows:\\
		\begin{figure*}[h]
	\centering
	\begin{tikzpicture}
	\draw[color=black] (0.0,0.0) rectangle (2.0,-0.5);
\node at (1.0,-0.25) {\color{black}{\scriptsize Warnaar}};
\node at (1.0,0.25) {\color{blue}{4}};
\draw[color=black] (2.3,0.0) rectangle (4.5,-0.5);
\node at (3.4,-0.25) {\color{black}{\scriptsize Theorem \ref{main-thm}}};
\draw[color=black] (5.7,0.0) rectangle (8.9,-0.5);
\node at (7.3,-0.25) {\color{black}{\scriptsize Tarasov-Varchenko}};
\draw[color=black] (9.1,0.0) rectangle (12.1,-0.5);
\node at (10.6,-0.25) {\color{black}{\scriptsize Dotsenko-Fateev}};
\draw[color=black] (2.3,-2.0) rectangle (4.5,-2.5);
\node at (3.4,-2.25) {\color{black}{\scriptsize Proposition \ref{cor:1}}};
\draw[color=black] (5.0,-2.0) rectangle (6.7,-2.5);
\node at (5.85,-2.25) {\color{black}{\scriptsize Selberg}};
\draw[color=black] (10.35,-2.00) rectangle (12.05,-2.50);
\node at (11.30,-2.25) {\color{black}{\scriptsize Mehta}};
\node at (12.30, -2.25) {\color{blue}{1}};\draw[->,>=angle 90,color=black] (3.4,-2.5) -- node {\color{black}{\scriptsize $\kern-1.5cm\ell=m=0$}} (3.4,-3.55);
\draw[->,>=angle 90,color=black] (3.40, -.50) -- node {} (3.40, -2.00) ;
\node at (3.40, -1.46) {\color{black}{\scriptsize $\kern-1.0cm x=1$}};
\draw[->,>=angle 90,color=black] (2.0,-0.5) -- node {} (5.0,-2.0) ;
\draw[->,>=angle 90,color=black] (6.18,-0.5) -- node {} (5.85,-2.0) ;
\draw[->,>=angle 90,color=black] (9.1,-0.5) -- node {} (6.615,-2.0) ;\node[draw,black,fill=white,circle,minimum size=0.3cm,inner sep=0pt] at (3.40, -4.00){\eqref{eqn:lm0}};
\node at (2.73, -4.00) {\color{blue}{3}};\draw[color=black] (.15,-6.35) rectangle (1.85,-5.65);
\node at (1.00,-6.00) {\color{black}{\scriptsize Example \ref{ex:2}}};
\node at (1.00,-6.55) {\color{blue}{2}};\draw[color=black] (2.55,-6.35) rectangle (4.25,-5.65);
\node at (3.40,-6.00) {\color{black}{\scriptsize Example \ref{ex:1}}};
\node at (3.40,-6.55) {\color{blue}{2}};\draw[color=black] (5.00,-6.35) rectangle (6.70,-5.65);
\node at (5.85,-6.00) {\color{black}{\scriptsize Example \ref{ex:3}}};
\node at (5.85,-6.55) {\color{blue}{2}};\draw[color=black] (8.05,-6.35) rectangle (9.75,-5.65);
\node at (8.90,-6.00) {\color{black}{\scriptsize Example \ref{ex:4}}};
\node at (8.90,-6.55) {\color{blue}{2}};\draw[color=black] (10.35,-6.35) rectangle (12.05,-5.65);
\node at (11.20,-6.00) {\color{black}{\scriptsize Example \ref{ex:5}}};
\node at (11.20,-6.55) {\color{blue}{1}};\draw[->,>=angle 90,color=black] (11.20, -2.50) -- node {} (11.20, -5.65);
\draw[->,>=angle 90,color=black] (1.00, -.50) -- node {} (1.00, -5.67);
\draw[->,>=angle 90,color=black] (3.02, -4.32) -- node {} (1.43, -5.64);
\draw[->,>=angle 90,color=black] (5.34, -2.50) -- node {} (3.59, -5.65);
\draw[->,>=angle 90,color=black] (3.40, -4.46) -- node {} (3.40, -5.64);
\draw[->,>=angle 90,color=black] (3.77, -4.30) -- node {} (5.38, -5.62);
\draw[->,>=angle 90,color=black] (3.84, -4.16) -- node {} (8.02, -5.68);
\draw[->,>=angle 90,color=black] (3.77, -4.30) -- node {} (5.38, -5.62);
\draw[->,>=angle 90,color=black] (10.30, -.50) -- node {} (9.00, -5.62);
\draw[->,>=angle 90,color=black] (7.62, -.50) -- node {} (5.97, -5.62);
\node[draw,black,fill=white,circle,minimum size=0.3cm,inner sep=0pt] at (8.52, -4.60) {\eqref{eqn:mehSpec}};
\node at (7.82, -4.60) {\color{blue}{1}};
\node at (7.78, -3.70) {\color{black}{\scriptsize$\lambda,\mu\to\infty$}};\draw[->,>=angle 90,color=black] (4.50, -2.50) -- node {} (8.12, -4.39);
\draw[->,>=angle 90,color=black] (8.96, -4.83) -- node {} (10.53, -5.65);
\node at (7.3,0.25) {\color{black}{\color{blue}{4}}};
\node at (10.6,0.25) {\color{black}{\color{blue}{3}}};
\node at (1.9,-2.25) {\color{black}{\color{blue}{3}}};
\node at (6.9,-2.25) {\color{black}{\color{blue}{3}}};
\node at (3.4,0.25) {\color{blue}{4}};
 				\end{tikzpicture}
		\caption{
	Specialization of Theorem \ref{main-thm}
	and related results.
    Blue numbers outside boxes denote the number
	of independent continuous parameters.
		}
				\label{fig:intdep}
			\end{figure*}
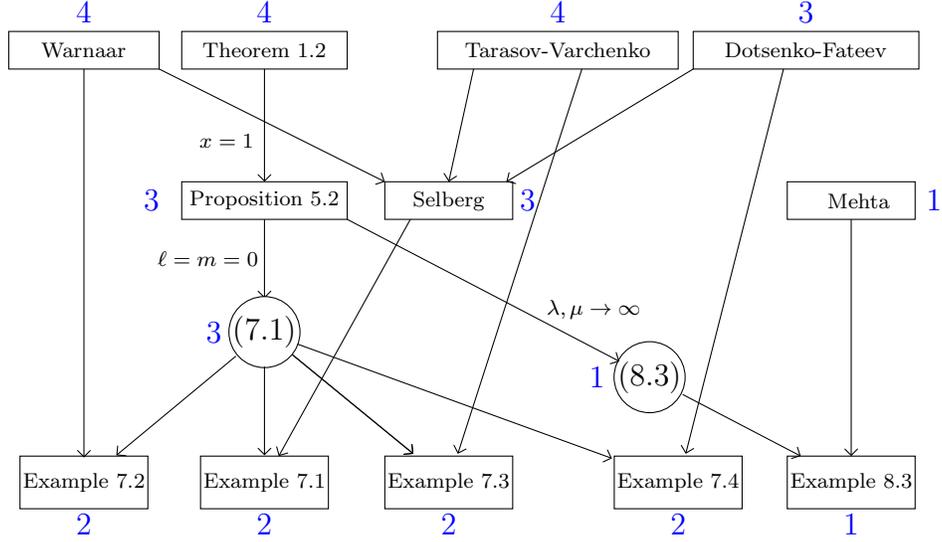

  	\section{Limiting case}\label{sec:limit}
In this section we discuss the limiting case of our integral formula.
Taking the limit in \eqref{eqn:s+t}
as both $\lambda$ and
$\mu$ tend to be
zero, we obtain
\begin{corollary}
	\label{cor:170599}For $\rho \in \mathbb{C}$ with ${\operatorname{Re}\rho} > 0$ and
  $\gamma \in \{ 0, 1 \}$,
  
  \begin{alignat}{1}
 &{\left|{\cos {\varphi}}+{\cos {\psi}}\right|^{\rho}{\operatorname{sgn}}^{\gamma}\left({\cos {\varphi}}+{\cos {\psi}}\right)}
 \nonumber
\\&=
{2^{-\rho}}{}
{{\Gamma\left( \rho+1 \right)}^{2}}{}
{\sum_{\mbox{\scalebox{0.7}        {$\begin{array}{c}\ell,m\in\mathbb{Z}\\\ell\equiv{m}+\gamma{\,\operatorname{mod}\,2}        \end{array}$}}}}
{\frac{{\cos {\ell{}\varphi}}{}{\cos {m{}\psi}}}
{{\prod_{\delta,\varepsilon\in{{\{\pm1\}}}}}
{{\Gamma\left( 1+\frac{1}{2}{}(\rho+\delta{}\ell+\varepsilon{}{m}) \right)}}}}
.\end{alignat}
 \end{corollary}
On the other hand, taking the limit in \eqref{eqn:cor:1} with $\varepsilon=1$ as $\lambda$ tends to infinity,
we can deduce the following integral formula of the Hermite polynomial $H_n(x)$ from
Proposition \ref{cor:1}:
\begin{corollary}
  \label{cor:Hermite}Suppose $\ell, m \in \mathbb{N}$ with $\ell + m$
  even.
  
  \begin{alignat}{1}
&{\displaystyle\int_{-\infty}^{\infty}\int_{-\infty}^{\infty}{\left|s-x{}{t}\right|^{2{}\nu}}{}{e^{-s^2-t^2}}{}{H_{\ell}\left(s\right)}{}{H_{m}\left(t\right)}{d}s{d}t}
\nonumber
    \\&=
    {\left(-\nu\right)_{\frac{\ell+m}{2}}}{}
    {\left(-1\right)^{\frac{\ell-m}{2}}}{}
    {2^{\ell+m}}{}
    {\pi^{\frac{1}{2}}}{}
    {\Gamma\left( \frac{1}{2}+\nu \right)}{}
    {\left(x^2+1\right)^{\nu-\frac{\ell+m}{2}}}{}
    {x^m}
    .
    \label{eqn:corHermite}
\end{alignat}
     \label{eqn:cor:Hermite}
\end{corollary}
\begin{proof}
	Use the limit formula
	\begin{equation*}
H_n (x) = n! \displaystyle\lim_{\lambda \rightarrow \infty} \lambda^{- \frac{n}{2}}
  C_n^{\lambda} \left( \displaystyle\frac{x}{\sqrt[]{\lambda}} \right).
	\end{equation*}
\end{proof}
\begin{example}
	\label{ex:5}
  (Mehta integral {\cite{mehta2004random}}) The Mehta integral
  \begin{equation*}
	  \frac{1}{(2 \pi)^{\frac{n}{2}}}
	  \displaystyle\int_{\mathbb{R}^n} \displaystyle\prod_{i = 1}^n e^{-\frac{1}{2}t_i^2} 
\displaystyle\prod_{1 \leqslant i < j \leqslant n} | t_i - t_j |^{2 \nu}
    d t_1 \cdots d t_n
     = \displaystyle\prod_{j = 1}^n \frac{\Gamma (1 + j \nu)}{\Gamma (1 + \nu)}
  \end{equation*}
  in the special case $n = 2$ implies the following equation
  
  \begin{equation}
      \frac{1}{2{}\pi}
      {\displaystyle\int_{-\infty}^{\infty}\int_{-\infty}^{\infty}{e^{-\frac{1}{2}(s^2+t^2)
      }}{}{\left|s-t\right|^{2{}\nu}}{d}s{d}t}=
    \frac{{\Gamma\left( 1+2{}\nu \right)}}{{\Gamma\left( 1+\nu \right)}} 
	  .\label{eqn:mehSpec}
  \end{equation}

   This coincides with the special case of Corollary \ref{cor:Hermite} with
  $(\ell, m, x) = (0, 0, 1)$.
\end{example}

\section*{Acknowledgement(s)}

The first author was partially supported by the Grant-in-Aid for Scientific Research (A) 
18H03669.

\nocite{gelfand1964}
\nocite{clerc2011}
\nocite{cohl2011}
\nocite{szmytkowski2011}

\end{document}